\documentclass{amsart}
\usepackage{amsmath, amssymb, mathrsfs, verbatim, multirow}

\newcommand{\mathsym}[1]{{}}

\newtheorem{theorem}{Theorem}[section]

\theoremstyle{definition}
\newtheorem{definition}[theorem]{Definition}

\newtheorem{rem}[theorem]{Remark}

\newtheorem{prop}[theorem]{Proposition}
\theoremstyle{remark}

\numberwithin{equation}{section}

\begin{document}

\title[infinite moment problem]{moment problem in infinitely many variables}

\author[M. Ghasemi, S. Kuhlmann, M. Marshall]{Mehdi Ghasemi$^1$, Salma Kuhlmann$^2$, Murray Marshall$^1$}

\address{$^1$Department of Mathematics and Statistics,\newline\indent
University of Saskatchewan,\newline\indent
Saskatoon, SK. S7N 5E6, Canada}
\email{mehdi.ghasemi@usask.ca, marshall@math.usask.ca}
\address{$^2$Fachbereich Mathematik und Statistik,\newline\indent
Universit\"{a}t Konstanz\newline\indent
78457 Konstanz, Germany}
\email{salma.kuhlmann@uni-konstanz.de}

\begin{abstract} The multivariate moment problem is investigated in the general context of the polynomial algebra 
$\mathbb{R}[x_i \mid i \in \Omega]$  in an arbitrary number of variables $x_i$, $i\in \Omega$. The results obtained 
are sharpest when the index set $\Omega$ is countable. Extensions of Haviland's theorem \cite{hav} and Nussbaum's 
theorem \cite{N} are proved. Lasserre's description of the support of the measure in terms of the non-negativity of 
the linear functional on a quadratic module of $\mathbb{R}[x_i \mid i \in \Omega]$  in \cite{L} is shown to remain 
valid in this more general situation.  The main tool used in the paper is an extension of the localization method 
developed by the third author in \cite{M1}, \cite{M2} and \cite{M3}. Various results proved in \cite{M1}, \cite{M2} 
and \cite{M3} are shown to continue to hold in this more general setting.
\end{abstract}

\maketitle

\section{Introduction}

The univariate moment problem is an old problem
with origins tracing back to work of Stieltjes \cite{St}. Given a sequence $(s_k)_{k\ge 0}$ of real numbers one wants 
to know when there exists a Radon measure $\mu$ on $\mathbb{R}$ such that 
\[
	s_k = \int x^k d\mu \ \forall \ k\ge 0.\footnote{All Radon measures considered are assumed to be positive.}
\]
Since the monomials $x^k$, $k\ge 0$ form a basis for the polynomial algebra $\mathbb{R}[x]$, this problem is equivalent 
to the following one: Given a linear functional $L : \mathbb{R}[x] \rightarrow \mathbb{R}$, when does there exist a Radon 
measure $\mu$ on $\mathbb{R}$ such that $L(f) = \int f d\mu$ $\forall$ $f\in \mathbb{R}[x]$. One also wants to know to 
what extent the measure is unique, assuming it exists. \cite{A} and \cite{ST} are standard references.

Work on the multivariate moment problem is more recent. For $n\ge1$, $\mathbb{R}[\underline{x}]:=\mathbb{R}[x_1,\dots,x_n]$ 
denotes the polynomial ring in $n$ variables $x_1,\dots,x_n$.
Given a linear functional $L : \mathbb{R}[\underline{x}] \rightarrow \mathbb{R}$ one wants to know when there exists a 
Radon measure $\mu$ on $\mathbb{R}^n$ such that $L(f) = \int f d\mu$ $\forall$ $f\in \mathbb{R}[\underline{x}]$. 
Again, one also wants to know to what extent the measure is unique, assuming it exists. \cite{B}, \cite{F}, \cite{L1}, 
\cite{M}, \cite{PS} are general references. A major motivation here is the close connection between the multivariate 
moment problem and polynomial optimization using semidefinite programming; see \cite{L1}, \cite{La}, \cite{M} and the 
references therein.

There appear to be only just a few papers dealing with the moment problem in infinitely many variables. Of these, the present 
paper seems to be the only one dealing with the general case. \cite{GIKM} deals with the special case where the linear 
functional $L : \mathbb{R}[x_i \mid i \in \Omega] \rightarrow \mathbb{R}$ is continuous with respect to a submultiplicative 
seminorm (more generally, with respect to a locally multiplicatively convex topology) on $\mathbb{R}[x_i \mid i \in \Omega]$. 
\cite{BCR}, \cite{GKM} and \cite{GMW} are precursors of \cite{GIKM}. \cite{BK} and \cite{IKR} deal  with the case of symmetric 
algebras over nuclear spaces.

The method used in the present paper is an extension of the localization method in \cite{M1}, \cite{M2} and \cite{M3}, 
the latter method being motivated in turn by results in \cite{KM}, \cite{M-1} and \cite{PV}. 
It is worth noting that, although some of the results in \cite{M1}, \cite{M2} and \cite{M3}  are similar in nature to 
those in \cite{PV}, the arguments are completely different.

The paper was written with no particular application in mind. At the same time, it seems reasonable to expect that 
applications do exist.  E.g., there may be connections to some variant of the semi-infinite polynomial optimization problem 
considered in \cite{L2}, \cite{WG}.


Section 2 introduces terminology and notation. Two important new concepts, constructibly Borel sets and constructibly Radon 
measures, are defined in this section.
In Section 3 we introduce three algebras $A =A_{\Omega}:= \mathbb{R}[x_i \mid i \in \Omega]$, 
$B= B_{\Omega} := \mathbb{R}[x_i, \frac{1}{1+x_i^2} \mid i\in \Omega]$, and 
$C=C_{\Omega} := \mathbb{R}[\frac{1}{1+x_i^2}, \frac{x_i}{1+x_i^2} \mid i\in \Omega]$, show how the moment problem for 
$A_{\Omega}$ reduces to understanding the extensions of a linear functional $L : A_{\Omega} \rightarrow \mathbb{R}$ 
to a PSD linear functional on $B_{\Omega}$, see Proposition \ref{extendibility}, and prove that PSD linear functionals 
$L : B_{\Omega} \rightarrow \mathbb{R}$ correspond bijectively to constructibly Radon measures on $\mathbb{R}^{\Omega}$, 
see Proposition \ref{new}. We also consider the important question of when the constructibly Radon measures thus obtained 
are actually Radon, see Proposition \ref{addendum}, Remark \ref{question} and Proposition \ref{radon}. In Section 4 we explain 
how results in \cite{M2} and \cite{M3} carry over, more-or-less word-for-word, to the case of infinitely many variables. 
In particular, we extend results of Fuglede \cite{F} and Petersen \cite{Pet}, see Propositions \ref{density corollary} and 
\ref{at most one}, and we establish extensions of Nussbaum's well-known sufficient condition for a linear functional 
$L : A_{\Omega} \rightarrow \mathbb{R}$ to correspond to a measure \cite{N}; see Propositions \ref{exactly one}, \ref{Nussbaum} 
and \ref{strong}. Section 5 deals with the problem of describing supporting sets for the measure. There are a number of 
important results in Section 5, see for example Proposition \ref{hav}, which is an extension of Haviland's theorem \cite{hav}, 
and Proposition \ref{support}, which is an extension of a result of Lasserre \cite{L}. In Section 6 we explain how the 
cylinder results in \cite{M1}, \cite{M2} and \cite{M3} extend to infinitely many variables.

The reader will notice that everything works more-or-less perfectly in case $\Omega$ is countable. If $\Omega$ is uncountable 
everything still works, but one typically only knows that the measures obtained are constructibly Radon (as opposed to Radon) 
and the results obtained concerning the support of the measure are a bit more restrictive than one might like.

\section{Terminology and Notation}

All rings considered are commutative with 1. All ring homomorphisms considered send 1 to 1. All rings we are interested in 
are $\mathbb{R}$-algebras. For $n\ge 1$, $\mathbb{R}[\underline{x}]:=\mathbb{R}[x_1,\dots,x_n]$.
For a topological space $X$, $C(X)$ denotes the ring of all continuous functions from $X$ to $\mathbb{R}$.

Let $A$ be a commutative ring. $X(A)$ denotes the \it character space \rm of $A$, i.e., the set of all ring homomorphisms 
$\alpha : A \rightarrow \mathbb{R}$.
For $a \in A$, $\hat{a} = \hat{a}_{A} : X(A) \rightarrow \mathbb{R}$ is defined by $\hat{a}_A(\alpha) = \alpha(a)$.
$X(A)$ is given the weakest topology such that the functions $\hat{a}_A$, $a \in A$ are continuous. 
The mapping $a \mapsto \hat{a}_A$ defines a ring homomorphism from $A$ into C$(X(A))$.
The only ring homomorphism from $\mathbb{R}$ to itself is the identity. Ring homomorphisms from $\mathbb{R}[\underline{x}]$ 
to $\mathbb{R}$ correspond to point evaluations $f \mapsto f(\alpha)$, $\alpha \in \mathbb{R}^n$.
$X(\mathbb{R}[\underline{x}])$ is identified (as a topological space) with $\mathbb{R}^n$.
By a \it quadratic module \rm of $A$ we mean a subset $M$ of $A$ satisfying 
\[
	1 \in M, \ M+M \subseteq M \text{ and } a^2M \subseteq M \text{ for each } a \in A.
\]
A \it preordering \rm of $A$ is a quadratic module of $A$ which is also closed under multiplication. For a subset $X$ of $X(A)$, 
\[
	\operatorname{Pos}_A(X) := \{ a \in A \mid \hat{a}_A \ge 0 \text{ on } X\}
\]
is a preordering of $A$.
We denote by $\sum A^2$ the set of all finite sums $\sum a_i^2$, $a_i \in A$. $\sum A^2$ is the unique smallest quadratic 
module of $A$. $\sum A^2$ is closed under multiplication, so $\sum A^2$ is also the unique smallest quadratic preordering of $A$.
For a subset $S \subseteq A$, the quadratic module of $A$ generated by $S$ consists of all finite sums 
$s_0+s_1g_1+\dots + s_kg_k$, $g_1,\dots, g_k \in S$, $s_0,\dots,s_k \in \sum A^2$. Also, 
\[
	X_S:= \{ \alpha \in X(A) \mid \hat{a}_A(\alpha)\ge 0 \ \forall a\in S\}.
\]
If $M= \sum A^2$ then $X_M = X(A)$. If $M$ is the quadratic module of $A$ generated by $S$ then $X_M:=  X_S$.
A quadratic module $M$ in $A$ is said to be \it archimedean \rm if for each $a \in A$ there exists an integer $k$ such that 
$k \pm a \in M$. If $M$ is a quadratic module of $A$ which is archimedean then $X_M$ is compact.
The converse is false in general \cite{JP}.

For simplicity, we assume from now on that $A$ is an $\mathbb{R}$-algebra. We record the following special case of the 
representation theorem of T. Jacobi \cite{J}.
\begin{prop}\label{jacobi} Suppose $M$ is an archimedean quadratic module of $A$. Then, for any $a \in A$, the following are equivalent:
\begin{itemize}
	\item[(1)] $\hat{a}_A \ge 0$ on $X_M$.
	\item[(2)] $a+\epsilon \in M$  for all real $\epsilon >0$.
\end{itemize}
\end{prop}

Note: The implication (2) $\Rightarrow$ (1) is trivial. The implication (1) $\Rightarrow$ (2) is non-trivial.
See \cite{BS}, \cite{K} and \cite{P} for early versions of Jacobi's theorem. See \cite{M} for a simple proof.

The open sets $$U_A(a) := \{ \alpha \in X(A)\mid \hat{a}_A(\alpha)>0\}, \ a\in A$$ form a subbasis for the topology on 
$X(A)$ (even a basis). Suppose $A$ is generated as an $\mathbb{R}$-algebra by $x_i$, $i\in \Omega$.
The embedding $X(A) \hookrightarrow \mathbb{R}^{\Omega}$ defined by $\alpha \mapsto (\alpha(x_i))_{i\in \Omega}$ identifies 
$X(A)$ with a subspace of $\mathbb{R}^{\Omega}$. Sets of the form 
\[
	\{ b \in \mathbb{R}^{\Omega} \mid \sum_{i\in I} (b_i-p_i)^2<r\},
\] 
where $r,p_i\in \mathbb{Q}$ and $I$ is a finite subset of $\Omega$, form a basis for the product topology on $\mathbb{R}^{\Omega}$.  
It follows that sets of the form
\begin{equation}\label{e}
U_A(r-\sum_{i\in I}(x_i-p_i)^2), \ r, p_i \in \mathbb{Q}, \ I \text{ a finite subset of } \Omega,
\end{equation}
form a basis for $X(A)$.

A subset $E$ of $X(A)$ is called \it Borel \rm if $E$ is an element of the $\sigma$-algebra of subsets of $X(A)$ generated by 
the open sets.  A subset $E$ of $X(A)$ is said to be \it constructible \rm (resp., \it constructibly Borel\rm) if $E$ is an 
element of the algebra (resp., $\sigma$-algebra) of subsets of $X(A)$ generated by the $U_A(a)$, $a \in A$.

Clearly constructible $\Rightarrow$ constructibly Borel $\Rightarrow$ Borel.
\begin{prop}\label{countable} If $A$ is generated as an $\mathbb{R}$-algebra by a countable set $\{ x_i \mid i\in \Omega\}$ 
then sets of the form (\ref{e}) form a countable basis for the topology on $X(A)$ and every Borel set of $X(A)$ is 
constructibly Borel.
\end{prop}
\begin{proof} This is clear.
\end{proof}
\begin{prop} A subset $E$ of $X(A)$ is constructibly Borel iff $E = \pi^{-1}(E')$ for some Borel set $E'$ of $X(A')$, where 
$A'$ is a countably generated subalgebra of $A$ and $\pi : X(A) \rightarrow X(A')$ is the canonical restriction map, i.e., 
$\pi(\alpha) =\alpha|_{A'}$.
\end{prop}
\begin{proof} Clearly $U_A(a) = \pi^{-1}(U_{A'}(a))$ for any $a\in A'$. Coupled with Proposition \ref{countable} this implies 
that, for each Borel set $E'$ of $X(A')$, $\pi^{-1}(E')$ is an element of the $\sigma$-algebra of subsets of $X(A)$ generated 
by the $U_A(a)$, $a\in A'$ (and conversely). Denote this $\sigma$-algebra by $\Sigma_{A'}$. It remains now to show that the 
union of the $\sigma$-algebras $\Sigma_{A'}$, $A'$ running through the countably generated subalgebras of $A$, is itself a 
$\sigma$-algebra. This follows from the well-known fact that a countable union of countable sets is countable (so the subalgebra 
of $A$ generated by countably many countably generated subalgebras of $A$ is itself countably generated).
\end{proof}
The support of a measure is not defined in general. For a measure space $(X, \Sigma, \mu)$ and a subset $Y$ of $X$, we say $\mu$ 
\it is supported by \rm $Y$ if $E\cap Y = \emptyset$ $\Rightarrow$ $\mu(E) = 0$ $\forall$ $E\in \Sigma$. In this situation, if 
$\Sigma' := \{ E\cap Y \mid E \in \Sigma\}$, and $\mu'(E\cap Y) := \mu(E)$ $\forall$ $E \in \Sigma$, then $\Sigma'$ is a 
$\sigma$-algebra of subsets of $Y$, $\mu'$ is a well-defined measure on $(Y,\Sigma')$, the inclusion map $i: Y \rightarrow X$ 
is a measurable function, and $\mu$ is the pushforward of $\mu'$ to $X$.

Recall that if $(Y,\Sigma',\mu')$ is a measure space, $(X,\Sigma)$ is a $\sigma$-algebra, $i : Y \rightarrow X$ is any measurable 
function, and $\mu$ is the pushforward of $\mu'$ to $(X,\Sigma)$, then for each measurable function $f : X \rightarrow \mathbb{R}$, 
$\int f d\mu = \int (f\circ i) d\mu'$ \cite[Theorem 39C]{H}. This is the well-known \it change in variables theorem. \rm

A \it Radon measure \rm on $X(A)$ is a positive measure $\mu$ on the $\sigma$-algebra of Borel sets of $X(A)$ which is locally 
finite and inner regular. Locally finite means that every point has a neighbourhood of finite measure.  Inner regular means each 
Borel set can be approximated from within using a compact set.
\begin{definition} A \it constructibly Radon measure \rm on $X(A)$ is a positive measure $\mu$ on the $\sigma$-algebra of constructibly 
Borel sets of $X(A)$ such that for, each countably generated subalgebra $A'$ of $A$, the pushforward of $\mu$ to $X(A')$ via the
restriction map $\alpha \mapsto \alpha|_{A'}$ is a Radon measure on $X(A')$.
\end{definition}
We are interested here in Radon and constructibly Radon measures having the additional property that $\hat{a}_A$ is $\mu$-integrable 
(i.e., $\int \hat{a}_A d\mu$ is well-defined and finite) for all $a\in A$.
For a linear functional $L : A \rightarrow \mathbb{R}$, one can consider the set of Radon or constructibly Radon measures $\mu$ on 
$X(A)$ such that $L(a) = \int \hat{a}_A  d\mu$ $\forall$ $a\in A$.
The \it moment problem \rm is to understand this set of measures, for a given linear functional $L : A \rightarrow \mathbb{R}$. 
In particular, one wants to know: (i) When is this set non-empty?
(ii) In case it is non-empty, when is it a singleton set?

A linear functional $L : A \rightarrow \mathbb{R}$ is said to be \it positive semidefinite \rm (PSD) if $L(\sum A^2) \subseteq [0,\infty)$ 
and \it positive \rm if $L(\operatorname{Pos}_A(X(A))\subseteq [0,\infty)$.

\section{Three special $\mathbb{R}$-algebras}

Let $A = A_{\Omega}:= \mathbb{R}[x_i \mid i \in \Omega]$, the ring of polynomials in the variables $x_i$, $i\in \Omega$ with 
coefficients in $\mathbb{R}$, $B= B_{\Omega} := \mathbb{R}[x_i, \frac{1}{1+x_i^2} \mid i\in \Omega]$, the localization of $A$ at 
the multiplicative set generated by the $1+x_i^2$, $i\in \Omega$, and 
$C = C_{\Omega} := \mathbb{R}[\frac{1}{1+x_i^2}, \frac{x_i}{1+x_i^2} \mid i\in \Omega]$, the $\mathbb{R}$-subalgebra of $B$ 
generated by the elements $\frac{1}{1+x_i^2}$, $\frac{x_i}{1+x_i^2}$, $i\in \Omega$. Here, $\Omega$ is an arbitrary index set.

By definition, $A$ (resp., $B$, resp., $C$) is the direct limit of the $\mathbb{R}$-algebras $A_I$ (resp., $B_I$, resp., $C_I$),
$I$ running through all finite subsets of $\Omega$. Because of this, many questions about $A$, $B$ and $C$ reduce immediately to 
the case where $\Omega$ is finite. Observe also that if $\Omega$ is finite then $B$ is equal to the localization of $A$ at 
$p : = \prod_{i\in \Omega} (1+x_i^2)$ considered in \cite{M2}. This is clear.

Elements of $X(A)$ and $X(B)$ are naturally identified with point evaluations $f \mapsto f(\alpha)$, $\alpha\in \mathbb{R}^{\Omega}$. 
Note that $X(A) = X(B)=\mathbb{R}^{\Omega}$, not just as sets, but also as topological spaces, giving $\mathbb{R}^{\Omega}$ the 
product topology.
\begin{prop}~\begin{enumerate}
	\item For $f\in B$, the following are equivalent:
	\begin{itemize}
		 \item[(i)]$f\in C$;
		 \item[(ii)] $f$ is \it geometrically bounded, \rm i.e., $\exists$ $k \in \mathbb{N}$ such that 
		 $|f(\alpha)| \le k$ $\forall$ $\alpha \in \mathbb{R}^{\Omega}$;
		 \item[(iii)] $f$ is \it algebraically bounded, \rm i.e., $\exists$ $k\in \mathbb{N}$ such that $k\pm f \in \sum B^2$.
	\end{itemize}
	\item $\sum B^2\cap C = \sum C^2$. In particular, $\sum C^2$ is archimedean.
	\item $C$ is naturally identified (via $y_i \leftrightarrow \frac{1}{1+x_i^2}$ and $z_i \leftrightarrow \frac{x_i}{1+x_i^2}$) 
	with the polynomial algebra $\mathbb{R}[y_i,z_i \mid i\in \Omega]$ factored by the ideal generated by the polynomials 
	$y_i^2+z_i^2-y_i = (y_i-\frac{1}{2})^2+z_i^2-\frac{1}{4}$, $i\in \Omega$. Consequently, $X(C)$ is identified naturally with 
	$\mathbb{S}^{\Omega}$, where $\mathbb{S} := \{ (y,z) \in \mathbb{R}^2 \mid (y-\frac{1}{2})^2+z^2 = \frac{1}{4}\}$.
	\item The restriction map $\alpha \mapsto \alpha|_C$ identifies $X(B)$ with a subspace of $X(C)$. In terms of coordinates, 
	this map is given by $\alpha = (x_i)_{i\in \Omega} \mapsto \beta = (y_i,z_i)_{i\in \Omega}$, where 
	$y_i := \frac{1}{1+x_i^2}$, $z_i:= \frac{x_i}{1+x_i^2}$. In particular, the image of $X(B)$ is dense in $X(C)$.
\end{enumerate}
\end{prop}
\begin{proof} 
See \cite[Remark 2.2]{M2}.
\end{proof}
\begin{prop}\label{positivstellensatz}~
\begin{itemize}
	\item[(1)] For $f\in C$, $\hat{f}_C\ge0$ on $X(C)$ iff $f+\epsilon \in \sum C^2$ $\forall$ real $\epsilon >0$.
	\item[(2)] For $f\in B$, $\hat{f}_B \ge 0$ on $X(B)$ iff $\exists$ $q$ of the form $q= \prod_{k=1}^n (1+x_{i_k}^2)^{\ell_k}$, 
	where $x_{i_1},\dots,x_{i_n}$ are the variables appearing in $f$, such that $f+\epsilon q \in \sum B^2$, $\forall$ real $\epsilon >0$.
\end{itemize}
\end{prop}
\begin{proof} (1) Since $\sum C^2$ is archimedean, this is immediate from proposition \ref{jacobi}.
(2) If $f\in B$, say $f\in B_{\{i_1,\dots,i_n\}}$, there exists an element $q$ of the form $q = \prod_{k=1}^n (1+x_{i_k}^2)^{\ell_k}$ 
such that $\frac{f}{q}\in C$. Thus, if $f \ge 0$ on $X(B)$ then $\frac{f}{q}\ge 0$ on $X(C)$ so $\frac{f}{q}+\epsilon \in \sum C^2$ and,
consequently, $f+\epsilon q \in \sum B^2$, $\forall$ real $\epsilon >0$.
\end{proof}
It follows from Proposition \ref{positivstellensatz} that linear functionals $L : C \rightarrow \mathbb{R}$, resp., 
$L: B \rightarrow \mathbb{R}$, are PSD iff they are positive. For linear functionals $L : A \rightarrow \mathbb{R}$ this is never 
the case, except when $|\Omega| \le 1$; see \cite{BCJ} \cite{S0}.
\begin{prop}\label{general extendibility}
Suppose $L : A \rightarrow \mathbb{R}$ is a linear functional and $L(\operatorname{Pos}_A(Y)) \subseteq [0,\infty)$ for some 
closed set $Y \subseteq \mathbb{R}^{\Omega}$. Then $L$ extends to a linear functional $L : B \rightarrow \mathbb{R}$ such that 
$L(\operatorname{Pos}_B(Y)) \subseteq [0,\infty)$.
\end{prop}
Note: The extension is not unique, in general.
\begin{proof}
The proof is a simple modification of the Zorn's lemma argument in \cite[Theorem 3.1]{M1}. Denote by $C'(Y)$ the $\mathbb{R}$-algebra 
of all continuous functions $f : Y \rightarrow \mathbb{R}$ which are bounded by some $\hat{a}$, $a\in A$ in the sense that 
$|f|\le |\hat{a}|$ on $Y$. As in the proof of \cite[Theorem 3.1]{M1} $\exists$ a linear functional 
$\overline{L} : C'(Y) \rightarrow \mathbb{R}$ which is positive (i.e., $\forall$ $f\in C'(Y)$, $f\ge 0$ on $Y$ $\Rightarrow$ 
$\overline{L}(f)\ge 0$) such that $L(a)=\overline{L}(\hat{a}|_Y)$ $\forall$ $a\in A$. If $b\in B$ then 
\[
	b = \frac{a}{(1+x_{i_1}^2)^{\ell_1}\cdots (1+x_{i_n}^2)^{\ell_n}}
\]
for some $a\in A$, $n\ge 1$, $i_k \in \Omega$, $\ell_k \ge 0$, $k=1,\dots,n$. Since $1+\alpha^2 \ge 1$ $\forall$ $\alpha \in \mathbb{R}$ 
it follows that $|\hat{b}| \le |\hat{a}|$ on $Y$, i.e., $\hat{b}|_Y \in C'(Y)$ $\forall$ $b\in B$. Define $L: B \rightarrow \mathbb{R}$ by 
$L(b)=\overline{L}(\hat{b}|_Y)$.
\end{proof}
\begin{prop}\label{extendibility} 
For a linear functional $L : A \rightarrow \mathbb{R}$ the following are equivalent:
\begin{itemize}
	\item[(1)] $L$ is a positive linear functional.
	\item[(2)] $L$ extends to a positive (i.e., PSD) linear functional $L : B \rightarrow \mathbb{R}$.
	\item[(3)] $\forall$ $f\in A$ and $\forall$ $p$ of the form $p= \prod_{k=1}^n (1+x_{i_k}^2)^{\ell_k}$, where 
	$x_{i_1},\dots,x_{i_n}$ are the variables appearing in $f$ and $\ell_k \ge 0$, $k=1,\dots,n$, $pf \in \sum A^2$ $\Rightarrow$ $L(f)\ge 0$.
\end{itemize}
\end{prop}
\begin{proof} 
(1) $\Rightarrow$ (2). Apply Proposition \ref{general extendibility} with $Y = \mathbb{R}^{\Omega}$.

(2) $\Rightarrow$ (3). Since $pf \in \sum A^2$, it follows that $p^2f \in \sum A^2$, so $f \in \sum B^2$. Since the extension of 
$L$ to $B$ is PSD this implies $L(f)\ge 0$.

(3) $\Rightarrow$ (1). Suppose $f\in A$, $\hat{f}\ge 0$. By Proposition \ref{positivstellensatz} (2) 
$\exists$ $q = \prod_{k=1}^n (1+x_{i_k}^2)^{\ell_k}$, where $x_{i_1},\dots,x_{i_n}$ are the variables appearing in $f$, such that 
$f+\epsilon q \in \sum B^2$ $\forall$ $\epsilon > 0$. Clearing denominators, $p^2(f+\epsilon q)\in \sum A^2$ for some $p$ 
(depending on $\epsilon$) of the form $p= \prod_{k=1}^n (1+x_{i_k}^2)^{m_k}$. By (3), $L(f)+\epsilon L(q)=L(f+\epsilon q)\ge 0$. 
Since $\epsilon>0$ is arbitrary, this implies $L(f)\ge 0$.
\end{proof}
PSD linear functionals $L : B \rightarrow \mathbb{R}$ restrict to PSD linear functionals on $C$. PSD linear functionals 
$L : C \rightarrow \mathbb{R}$ are in natural one-to-one correspondence with Radon measures $\mu$ on the compact space $X(C)$ via 
$L \leftrightarrow \mu$ iff $L(f) = \int \hat{f}_C d\mu$ $\forall$ $f\in C$. This is well-known, e.g., see \cite[Corollary 3.3 and 
Remark 3.5]{M1}.

For $i \in \Omega$, let $\Delta_i := \{ \beta \in X(C) \mid \beta(\frac{1}{1+x_i^2})=0\}.$ Because of the way $X(C)$ is being 
identified with $\mathbb{S}^{\Omega}$, $\Delta_i$ is identified with the set 
\[
	\{ (y_j,z_j)_{j\in \Omega} \in \mathbb{S}^{\Omega} \mid y_i= z_i =0\}.
\]
It is clear that $X(C)\backslash X(B)$ is the union of the sets $\Delta_i$, $i\in \Omega$.
For each $f\in B$ one can associate a continuous  function 
\[
	\tilde{f} : X(C)\backslash (\Delta_{i_1}\cup \dots \cup \Delta_{i_n}) \rightarrow \mathbb{R},
\]
where $x_{i_1},\dots,x_{i_n}$ are the variables appearing in $f$. Observe that $f \in B_{\{i_1,\dots,i_n\}}$. 
Define $\tilde{f} = \hat{f}_{B_{\{i_1,\dots,i_n\}}}\circ \pi$ where $\pi : X(C) \rightarrow X(C_{\{i_1,\dots,i_n\}})$ is the 
restriction map. Observe that the inverse image under $\pi$ of the set $X(C_{\{i_1,\dots,i_n\}})\backslash X(B_{\{i_1,\dots,i_n\}})$ 
is precisely the set $\Delta_{i_1}\cup \dots \cup \Delta_{i_n}$. Note also that $\tilde{f}|_{X(B)} = \hat{f}_B$.
\begin{prop}\label{main} 
For each PSD linear functional $L : B \rightarrow \mathbb{R}$ there exists a unique Radon measure $\mu$ on $X(C)$ such that 
$L(f) = \int \hat{f}_C d\mu$ $\forall$ $f\in C$. This satisfies $\mu(\Delta_i)=0$ $\forall$ $i\in \Omega$ and 
$L(f) = \int \tilde{f} d\mu$ $\forall$ $f\in B$.
\end{prop}
\begin{proof} 
Fix a finite subset $I =\{ i_1,\dots, i_n\}$ of $\Omega$. By \cite[Corollary 3.3]{M1} there exists a Radon measure $\mu$ on 
$X(C)$ and a Radon measure $\mu_I$ on $X(C_I)$ such that $L(f) = \int \hat{f}_C d\mu$ $\forall$ $f \in C$ and 
$L(f) = \int \hat{f}_{C_I} d\mu_I$ $\forall$ $f \in C_I$. Applying \cite[Corollary 3.4]{M1} with $p = (1+x_{i_1}^2)\dots (1+x_{i_n}^2)$, 
there exists a Radon measure $\nu_I$ on $X(B_I)$ such that $L(f) = \int \hat{f}_{B_I} d\nu_I$ $\forall$ $f\in B_I$. 
By \cite[Remark 3.5]{M1} the measures $\mu$, $\mu_I$, $\nu_I$ are unique. Denote by $\mu_I'$ the pushforward of $\mu$ to $X(C_I)$ 
by the restriction map $\pi : X(C) \rightarrow X(C_I)$. Since $\hat{f}_C = \hat{f}_{C_I}\circ \pi$ $\forall$ $f\in C_I$, it follows 
that $\int \hat{f}_{C_I} d\mu_I' = \int \hat{f}_C d\mu = L(f)$ $\forall$ $f\in C_I$, so uniqueness of $\mu_I$ implies $\mu_I' = \mu_I$. 
A similar argument shows that $\mu_I$ is the pushforward of $\nu_I$ via the natural embedding $X(B_I) \hookrightarrow X(C_I)$. 
It follows that $\mu(\Delta_{i_1} \cup \dots \cup \Delta_{i_n}) = \mu_I(X(C_I)\backslash X(B_I)) =0$. Since $I$ is an arbitrary 
finite subset of $\Omega$, this implies $\mu(\Delta_i)=0$ $\forall$ $i\in \Omega$. Suppose now that $f \in B_I$. 
Since $\tilde{f} = \hat{f}_{B_I}\circ \pi$,  $\int \tilde{f} d\mu = \int \hat{f}_{B_I} d\nu_I = L(f)$ as required.
\end{proof}
One would like to know when there exists a Radon measure $\nu$ on $X(B)$ such that $L(f) = \int \hat{f}_B d\nu$ $\forall$ $f\in B$.
\begin{prop}\label{addendum}
Let $L$ be a PSD linear functional on $B$ and let be $\mu$ be the Radon measure on $X(C)$ associated to $L$. 
The following are equivalent:
\begin{itemize}
	\item[(1)] $\exists$ a Radon measure $\nu$ on $X(B)$ such that $L(f) = \int \hat{f}_B d\nu$ $\forall$ $f\in B$.
	\item[(2)] $\forall$ Borel sets $E$ of $X(C)$, $\mu(E) = \sup\{ \mu(K) \mid K \text{ compact, } K\subseteq X(B)\cap E\}$.
	\item[(3)] $\mu(X(C)) = \sup\{ \mu(K) \mid K \text{ compact, } K \subseteq X(B)\}$.
	\item[(4)] $\mu$ is supported by a Borel set $E$ of $X(C)$ such that $E \subseteq X(B)$.
\end{itemize}
Moreover, if this is the case, then $\nu(E) = \sup\{ \mu(K) \mid K \text{ compact, } K\subseteq E\}$ $\forall$ Borel sets 
$E$ of $X(B)$. In particular, $\nu$ is uniquely determined by $\mu$.
\end{prop}
\begin{proof} Assume (1).
Denote by $\mu'$ the pushforward of $\nu$ to $X(C)$.
Then, $\forall$ $f\in C$, $\int \hat{f}_C d\mu' = \int \hat{f}_B d\nu = L(f)$.
Uniqueness of $\mu$ implies $\mu' = \mu$. Since $\nu$ is Radon, (2) is now clear. (2) $\Rightarrow$ (3) is obvious. Assume (3). 
Define $E = \cup_{n\ge 1} K_n$ where $K_n$ is a compact subset of $X(B)$ such that $\mu(X(C)\backslash K_n) \le \frac{1}{n}$. 
Clearly $E \subseteq X(B)$, $E$ is a Borel set of $X(C)$ and $\mu$ is supported by $E$. This proves (4).
Assume (4). Then $\nu$ defined by $\nu(E'\cap X(B)) = \mu(E')$ $\forall$ Borel sets $E'$ of $X(C)$ is a Radon measure on $X(B)$. 
Since $\mu$ is the pushforward of $\nu$ to $X(C)$, $\int \hat{f}_B d\nu = \int \tilde{f} d\mu = L(f)$, so (1) is clear.
The last assertion is clear.
\end{proof}
\begin{rem}\label{question} 
If $\Omega$ is countable then $X(C)\backslash X(B) = \cup_{i\in \Omega} \Delta_i$ is a Borel set of measure zero, so the equivalent 
conditions of Proposition \ref{addendum} hold in this case. We know of no example where the conditions of Proposition \ref{addendum} fail. 
It would be nice to have an example.\footnotemark\footnotetext{If we assume $\Omega$ is uncountable then it is easy enough to 
construct a Radon measure $\mu$ on $X(C)$ so that the equivalent conditions (2) and (3) fail. This is not a problem. The problem is to 
choose such a $\mu$ so that, in addition, $\int \tilde{f} d\mu$ is well-defined and finite for all $f \in B$.}
\end{rem}

It seems probable that, to handle the most general case, one has to relax the requirement that $\nu$ be Radon, requiring only that 
$\nu$ be constructibly Radon.
\begin{prop}\label{new}
There is a canonical one-to-one correspondence $L \leftrightarrow \nu$ given by $L(f) = \int \hat{f}_B d\nu$ $\forall$ $f\in B$ 
between PSD linear functionals $L$ on $B$ and constructibly Radon measures $\nu$ on $X(B)$ with the property that $\hat{f}_B$ is 
$\nu$-integrable $\forall$ $f\in B$.
\end{prop}
\begin{proof}
If $\nu$ is a constructibly Radon measure on $X(B)$ and $\hat{f}_B$ is $\nu$-integrable $\forall$ $f\in B$ then it is clear that 
the map $f \mapsto \int \hat{f}_B d\nu$, $f\in B$ is a PSD linear functional on $B$.
Conversely, suppose $L$ is a PSD linear functional on $B$. Let $\mu$ be the measure defined in Proposition \ref{main}. For each subset 
$I$ of $\Omega$, consider the subalgebra $B_I$ of $B$ and the subalgebra $C_I$ of $C$. Denote by $\mu_I$ the pushforward of $\mu$ 
via the canonical restriction map $\pi : X(C) \rightarrow X(C_I)$. One checks that $\mu_I$ is the Radon measure on $X(C_I)$ 
corresponding to the PSD linear map $L|_{B_I}$. In particular, if $I$ is countable then  $\mu_I(X(C_I)\backslash X(B_I))=0$.

\smallskip
\noindent
Claim 1: If $E$ in $X(C)$ is constructibly Borel and $X(B)\cap E = \emptyset$, then $\mu(E)=0$. Say $E= \pi^{-1}(E')$, $E'$ a Borel 
set in $X(C_I)$, $I\subseteq \Omega$ countable. Since the restriction map $X(B) \rightarrow X(B_I)$ is surjective, our hypothesis 
implies that $X(B_I)\cap E' = \emptyset$. It follows that $\mu(E) = \mu_I(E') = 0$ as required.

\smallskip
\noindent
Claim 2: The constructibly Borel sets in $X(B)$ are precisely the sets of the form $X(B)\cap E$ where $E$ is constructibly Borel 
in $X(C)$. This is more or less clear. If $f\in C$ then $U_B(f) = X(B)\cap U_C(f)$. If $f \in B$ then there exists $p$ of the form 
$p = \prod_{k=1}^n (1+x_{i_k}^2)^{\ell_k}$ where $x_{i_1},\dots,x_{i_n}$ are the variables appearing in $f$ such that $\frac{f}{p} \in C$. 
Also, $U_B(f) = U_B(\frac{f}{p})$ for any such $p$.

\smallskip

Define a measure $\nu$ on the $\sigma$-algebra of constructibly Borel subsets of $X(B)$ by $\nu(X(B)\cap E) := \mu(E)$ $\forall$ 
constructibly Borel subsets $E$ of $X(C)$. By Claim 1 $\nu$ is well-defined. By construction, $\mu$ is the pushforward of $\nu$ to $X(C)$. 
Also, $\hat{f}_B = \tilde{f}|_{X(B)}$ $\forall$ $f\in B$. 
It follows that $\int \hat{f}_B d\nu = \int \tilde{f} d\mu = L(f)$ $\forall$ $f\in B$. For each countable $I \subseteq \Omega$, 
the pushforward of $\nu$ to $X(B_I)$ is the Radon measure $\nu_I$ on $X(B_I)$ induced by $\mu_I$ using Proposition \ref{addendum} 
and Remark \ref{question}. It follows that $\nu$ is constructibly Radon.

It remains to show that $\nu$ is unique. Let $\nu'$ be any constructibly Radon measure on $X(B)$ such that 
$\int \hat{f}_B d\nu' = L(f)$ $\forall$ $f\in B$. For $I \subseteq \Omega$ countable, let $\nu'_I$ be the pushforward of $\nu'$ to 
$X(B_I)$ and let $\mu'_I$ be the pushforward of $\nu'_I$ to $X(C_I)$. 
Then $L(f) = \int \hat{f}_B d\nu' = \int \hat{f}_{C_I} d \mu'_I$ $\forall$ $f\in C_I$. Since we also have 
$L(f) = \int \hat{f}_B d\nu = \int \hat{f}_{C_I} d \mu_I$ $\forall$ $f\in C_I$, this implies 
$\int \hat{f}_{C_I} d\mu'_I = \int \hat{f}_{C_I} d\mu_I$ $\forall$ $f\in C_I$. 
Thus by uniqueness of $\mu_I$, $\mu'_I = \mu_I$ $\forall$ countable $I \subseteq \Omega$. This in turn implies that 
$\nu'_I = \nu_I$ $\forall$ countable $I \subseteq \Omega$, so $\nu' = \nu$.
\end{proof}
\begin{rem}\label{support remark} 
If $\mu$ is supported by a constructibly Borel set $K$ in $X(C)$ then $\nu$ is supported by $K\cap X(B)$. This follows from 
Claim 1. If $E$ is a constructibly Borel set in $X(C)$ and $E\cap K\cap X(B)=\emptyset$ then $\mu(E\cap K)=0$. Since $\mu$ is 
supported by $K$ this implies in turn that $\nu(E\cap X(B)) =\mu(E)=0$. Unfortunately, we are unable to prove this in the more 
general setting where $K$ is only assumed to be a Borel set of $X(C)$.
Of course, if $\mu$ happens to be the pushforward of a Radon measure $\nu$ on $X(B)$ (the case considered in Proposition 
\ref{addendum}) then $\mu$ supported by $K$ $\Rightarrow$ $\nu$ supported by $K\cap X(B)$ for any Borel set $K$ of $X(C)$.
\end{rem}
\section{Moment Problem}

We fix an index set $\Omega$ and define $A=A_{\Omega}$, $B=B_{\Omega}$ and $C=C_{\Omega}$ as in the previous section. 
We identify $X(A)=X(B) = \mathbb{R}^{\Omega}$. The measures $\nu$ arising in Proposition \ref{new} have \it finite moments, \rm 
i.e., $\int \hat{x^{\alpha}} d\nu$ is well-defined and finite for all monomials 
$x^{\alpha} := x_{i_1}^{\alpha_1}\dots x_{i_n}^{\alpha_n}$, $\{ i_1,\dots,i_n\} \subseteq \Omega$, $\alpha_k\ge 0$, $k=1,\dots,n$. 
Conversely, if $\nu$ is a constructibly Borel measure on $\mathbb{R}^{\Omega}$ having finite moments then $L : B \rightarrow \mathbb{R}$ 
defined by $L(f) := \int \hat{f} d\nu$ is a well-defined PSD linear functional on $B$. This is clear.

Much of what was done in \cite{M2} and \cite{M3} in the finite dimensional case carries over, more or less word for word, 
to the infinite dimensional case. Recall if $(X, \Sigma, \mu)$ is a measure space and $f : X \rightarrow \mathbb{C}$ is a 
measurable function, then 
\[
	\| f\|_{s,\mu} := \big[\int |f|^s d\mu \big]^{1/s}, \ s\in [1,\infty).
\] 
The \it Lebesgue space \rm $\mathcal{L}^s(\mu)$, by definition, is the $\mathbb{C}$-vector space  
\[
	\mathcal{L}^s(\mu) := \{ f : X \rightarrow \mathbb{C} \mid f \text{ is measurable and } \| f\|_{s,\mu} <\infty\}
\] 
equipped with the norm $\| \cdot \|_{s,\mu}$.
\begin{prop}\label{lp spaces} 
Suppose $\nu$ is a constructibly Radon measure on $\mathbb{R}^{\Omega}$ having finite moments. Then for any $s \in [1,\infty)$ 
the obvious $\mathbb{C}$-linear map $B\otimes_{\mathbb{R}} \mathbb{C} \rightarrow \mathcal{L}^s(\nu)$, $f\otimes r \mapsto r\hat{f}$ 
has dense image, equivalently, the image of $B$ under the $\mathbb{R}$-linear map $f \mapsto \hat{f}$ is dense in the real part of
$\mathcal{L}^s(\nu)$.
\end{prop}
Note that $A\otimes_{\mathbb{R}} \mathbb{C}= \mathbb{C}[x_i \mid i \in \Omega]$, 
$B\otimes_{\mathbb{R}} \mathbb{C}= \mathbb{C}[x_i, \frac{1}{1+x_i^2} \mid i\in \Omega]$,  and 
$C\otimes_{\mathbb{R}} \mathbb{C} = \mathbb{C}[\frac{1}{1+x_i^2}, \frac{x_i}{1+x_i^2} \mid i\in \Omega]$.

\begin{proof}
It suffices to show that the step functions $\sum_{j=1}^m r_j\chi_{E_j}$, $r_j \in \mathbb{C}$, $E_j \subseteq X(B)$ a 
constructibly Borel set, belong to the closure of the image of $B\otimes_{\mathbb{R}} \mathbb{C}$. Using the triangle inequality 
we are reduced further to the case $m=1$, $r_1=1$. Let $E \subseteq X(B)$ be a constructibly Borel set. Writing $E= \pi^{-1}(E')$ 
where $E'$ is a Borel set in $X(B_I)$, for some appropriate countable $I \subseteq \Omega$, and applying the change of variable 
theorem, we see that $\| \chi_E -\hat{f}_B\|_{s,\nu} = \|\chi_{E'}-\hat{f}_{B_I}\|_{s,\nu_I}$ $\forall$ $f\in B_I$. In this way 
we are reduced to the case where $\Omega$ is countable.  Choose $K$ compact, $U$ open in $X(C)$ such that $K \subseteq E \subseteq U$, 
$\mu(U\backslash K)< \epsilon$. By Urysohn's lemma there exists a continuous function $\phi : X(C) \rightarrow \mathbb{R}$ such that 
$0\le \phi \le 1$ on $X(C)$, $\phi = 1$ on $K$, $\phi = 0$ on $X(C) \backslash U$. Then $\| \chi_E-\phi\|_{s,\mu} \le \epsilon^{1/s}$. 
Use the Stone-Weierstrass approximation theorem to get $f\in C$ such that $\|\phi- \hat{f}_C\|_{\infty}<\epsilon$, where 
$\| \cdot \|_{\infty}$ denotes the sup-norm. Then $\|\phi- \hat{f}_C\|_{s,\mu} \le \epsilon\mu(X(C))^{1/s}$. 
Putting these things together yields 
$\|\chi_E-\hat{f}_B\|_{s,\nu}=\|\chi_E-\hat{f}_C\|_{s,\mu}\le\|\chi_E-\phi\|_{s,\mu}+\|\phi-\hat{f}_C\|_{s,\mu}\le \epsilon^{1/s}+\epsilon\mu(X(C))^{1/s}$.
\end{proof}
From now on, by a constructibly Radon measure on $\mathbb{R}^{\Omega}$ we will mean a constructibly Radon measure on 
$\mathbb{R}^{\Omega}$ having finite moments.
\begin{prop}
For any constructibly Radon measure $\nu$ on $\mathbb{R}^{\Omega}$ and any $s\in [1,\infty)$, 
$A_{\Omega}\otimes_{\mathbb{R}} \mathbb{C}$ is dense in $\mathcal{L}^s(\nu)$ iff $A_{\Omega}\otimes_{\mathbb{R}} \mathbb{C}$ 
is dense in $B_{\Omega}\otimes_{\mathbb{R}} \mathbb{C}$ in the $\| \cdot \|_{s,\nu}$-norm.
\end{prop}
\begin{proof} 
Since the density property in question is transitive, this is immediate from Proposition \ref{lp spaces}.
\end{proof}
\begin{prop}\label{density corollary} 
Suppose $\nu$ is a constructibly Radon measure on $\mathbb{R}^{\Omega}$ and $s\in (1,\infty)$. Suppose for each 
$j\in \Omega$ $\exists$ $q_{jk} \in A_{\Omega}\otimes_{\mathbb{R}} \mathbb{C}$ such that 
$\| q_{jk}-\frac{1}{x_j-\mathrm{i}}\|_{s,\nu} \rightarrow 0$ as $k\rightarrow \infty$. 
Then $A_{\Omega}\otimes_{\mathbb{R}} \mathbb{C}$ is dense in $\mathcal{L}^{s'}(\nu)$ for each $s' \in [1,s)$.
\end{prop}
\begin{proof}
Denote by $\overline{A_{\Omega}\otimes_{\mathbb{R}}\mathbb{C}}$ the closure of $A_{\Omega}\otimes_{\mathbb{R}}\mathbb{C}$ with 
respect to topology induced by the norm $\|\cdot \|_{s',\nu}$. It suffices to show that each 
$f \in B_{\Omega}\otimes_{\mathbb{R}}\mathbb{C}$ belongs to $\overline{A_{\Omega}\otimes_{\mathbb{R}}\mathbb{C}}$.
The proof is by induction on the number of factors of the form $x_j\pm\mathrm{i}$ appearing in the denominator of $f$. 
Suppose $x_j-\mathrm{i}$ appears in the denominator of $f$. By induction, $fq_{jk}(x_j-\mathrm{i})$ belongs to 
$\overline{A_{\Omega}\otimes_{\mathbb{R}}\mathbb{C}}$, for each $k \ge 1$. Applying H\"older's inequality 
\[
	\int |PQ| d\nu \le [\int |P|^a d\nu]^{\frac{1}{a}}\cdot [\int |Q|^b d\nu]^{\frac{1}{b}}, \ \frac{1}{a}+\frac{1}{b}=1
\]
with $P = |q_{jk}-\frac{1}{x_j-\mathrm{i}}|^{s'}$, $Q = |f(x_j-\mathrm{i})|^{s'}$, $a= \frac{s}{s'}$, $b= \frac{s}{s-s'}$, we see that
\begin{align*}
\| fq_{jk}(x-\mathrm{i})-f\|_{s',\nu} =& \| (q_{jk}-\frac{1}{x_j-\mathrm{i}})f(x_j-\mathrm{i})\|_{s',\nu} \\ \le& \| q_{jk}-\frac{1}{x_j-\mathrm{i}}\|_{s,\nu} \cdot \| f(x_j-\mathrm{i})\|_{\frac{ss'}{s-s'},\nu}.
\end{align*}
It follows that $f$ belongs to $\overline{A_{\Omega}\otimes_{\mathbb{R}}\mathbb{C}}$.
The case where $x_j+\mathrm{i}$ appears in the denominator of $f$ is dealt with similarly, replacing $q_{jk}$ by $\overline{q_{jk}}$.
\end{proof}
Proposition \ref{density corollary} extends Petersen's result in \cite[Proposition]{Pet}. In the one variable case, i.e., when 
$|\Omega| = 1$, one can conclude also that $A_{\Omega}\otimes_{\mathbb{R}} \mathbb{C}$ is dense in $\mathcal{L}^s(\nu)$; 
see \cite[Corollary 3.3]{M2}.

Caution: The proof given in \cite[Corollary 3.6]{M2} is not correct.
The proof in \cite[Corollary 3.6]{M2} is correct when $q_{jk} \in \mathbb{C}[x_j]$ for each $j$ and $k$.

The following result extends \cite[Corollary 2.5]{M2} to the case where $\Omega$ is infinite.
\begin{prop}
For any linear functional $L : A_{\Omega} \rightarrow \mathbb{R}$, the set of constructibly Radon measures $\nu$ on 
$\mathbb{R}^{\Omega}$ satisfying $L(f) = \int \hat{f} d\nu$ $\forall$ $f \in A_{\Omega}$ is in natural one-to-one correspondence 
with the set of PSD linear functionals $L' : B_{\Omega} \rightarrow \mathbb{R}$ extending $L$.
\end{prop}
\begin{proof} 
If $\nu$ is a constructibly Radon measure on $\mathbb{R}^{\Omega}$ such that $L(f) = \int \hat{f} d\nu$ $\forall$ $f\in A_{\Omega}$, 
the corresponding extension of $L$ to a PSD linear functional $L': B_{\Omega} \rightarrow \mathbb{R}$ is defined by 
$L'(f)= \int \hat{f} d\nu$ $\forall$ $f\in B_{\Omega}$. The correspondence $\nu \mapsto L'$ has the desired properties by Proposition \ref{new}.
\end{proof}
For $\nu$ any constructibly Radon measure on $\mathbb{R}^{\Omega}$, define $L_{\nu} : A_{\Omega}\rightarrow \mathbb{R}$ by 
$L_{\nu}(f) = \int \hat{f} d\nu$ $\forall$ $f\in A_{\Omega}$. 
If $\nu'$ is another constructibly Radon measure on $\mathbb{R}^{\Omega}$, we write $\nu \sim \nu'$ to indicate that $\nu$ and 
$\nu'$ have the same moments, i.e., $L_{\nu} = L_{\nu'}$. We say $\nu$ is \it determinate \rm if $\nu \sim \nu'$ $\Rightarrow$ $\nu = \nu'$ 
and \it indeterminate \rm if this is not the case.

\begin{prop}\label{at most one} 
Suppose $L : A_{\Omega} \rightarrow \mathbb{R}$ is linear and, for each $j \in \Omega$,
\begin{equation}\label{m}
\exists \text{ a sequence } \{p_{jk}\}_{k=1}^{\infty} \text{ in } A_{\Omega}\otimes \mathbb{C} \text{ such that }  \lim\limits_{k\rightarrow\infty}L(|1-(x_j-\mathrm{i})p_{jk}|^2)=0.
\end{equation}
Then there is at most one constructibly Radon measure $\nu$ on $\mathbb{R}^{\Omega}$ such that $L = L_{\nu}$.
\end{prop}
\begin{proof}
Argue as in \cite[Corollary 2.7]{M2}.
\end{proof}
Proposition \ref{at most one} extends Fuglede's result in \cite[Section 7]{F} and Petersen's result in \cite[Theorem 3]{Pet}.
\begin{prop}\label{exactly one} 
Suppose $L : A_{\Omega} \rightarrow \mathbb{R}$ is linear and PSD and, for each $j \in \Omega$,
\begin{equation}\label{n}
\exists \text{ a sequence } \{q_{jk}\}_{k=1}^{\infty} \text{ in } A_{\Omega}\otimes \mathbb{C} \text{ such that } \lim\limits_{k\rightarrow\infty}L(|1-(1+x_j^2)q_{jk}\overline{q_{jk}}|^2)=0.
\end{equation}
Then there exists a unique constructibly Radon measure $\nu$ on $\mathbb{R}^{\Omega}$ such that $L = L_{\nu}$.
\end{prop}
\begin{proof}
Argue as in \cite[Corollaries 4.7 and 4.8]{M2} and \cite[Theorem 0.1]{M3}.
\end{proof}
\begin{rem}\label{clarify}~
\begin{itemize}
	\item[(i)] For each $j\in \Omega$, condition (\ref{m}) is implied by condition (\ref{n}). 
	This is clear. Just take $p_{jk} = (x_j+\mathrm{i})q_{jk}\overline{q_{jk}}$.
    \item[(ii)] For each $j\in \Omega$, condition (\ref{n}) is implied by the Carleman condition
	\begin{equation}\label{o}
		\sum_{k=1}^{\infty} \frac{1}{\root 2k \of{L(x_j^{2k})}} = \infty.
	\end{equation}
	See \cite[Th\'eor\`eme 3]{BC} and \cite[Lemma 0.2 and Theorem 0.3]{M3} for the proof.
	\item[(iii)] The example in \cite{So} shows that (\ref{n}) is strictly weaker than (\ref{o}).
\end{itemize}
\end{rem}
Combining Proposition \ref{exactly one} and Remark \ref{clarify} (ii) yields the following result, which is an extension of 
Nussbaum's result in \cite{N}.
\begin{prop}\label{Nussbaum}
Suppose $L : A_{\Omega} \rightarrow \mathbb{R}$ is linear and PSD and, for each $j \in \Omega$, the Carleman condition (\ref{o}) holds.
Then there exists a unique constructibly Radon measure $\nu$ on $\mathbb{R}^{\Omega}$ such that $L = L_{\nu}$.
\end{prop}
\begin{rem} Condition (\ref{o}) holds in a large number of cases. It holds, for example, if there exists a constant $C_j >0$ such 
that $L(x_j^{2k})\le C_j(2k)!$ for all $k\ge 1$. It holds, in particular, if $L$ is continuous with respect to the vector space 
norm $\rho_w: A_{\Omega} \rightarrow [0,\infty)$ defined by $\rho_w(\sum a_{\alpha} x^{\alpha}) := \sum_{\alpha} |a_{\alpha}|w_{\alpha}$ 
where $w_{\alpha} := (2\lceil |\alpha|/2\rceil)!$, see \cite{L} for the definition of $\rho_w$ in case $|\Omega|<\infty$, or if $L$ 
is continuous with respect to the finest locally multiplicatively convex topology on $A_{\Omega}$, see \cite{GIKM} and \cite{GKM}.
\end{rem}
We mention another result of the same flavour which, in case $|\Omega|<\infty$, is due to Schm\"udgen; see \cite[Theorem 4.11]{M2} 
\cite[Proposition 1]{S}.

\begin{prop}\label{strong}
Suppose $L : A_{\Omega} \rightarrow \mathbb{R}$ is linear and PSD. For each $j\in \Omega$ fix a Radon measure $\mu_j$ on 
$\mathbb{R}$ such that $L|_{\mathbb{R}[x_j]} = L_{\mu_j}$ and suppose, for each $j \in \Omega$, that $\mathbb{C}[x_j]$ is dense in 
$\mathcal{L}^4(\mu_j)$, i.e.,
\begin{equation}\label{s}
\exists \text{ a sequence } \{q_{jk}\}_{k=1}^{\infty} \text{ in } \mathbb{C}[x_j] \text{ such that } \lim\limits_{k\rightarrow\infty}\|q_{jk}-\frac{1}{x_j-\mathrm{i}}\|_{4,\mu_j}=0.
\end{equation}
Then there exists a unique constructibly Radon measure $\mu$ on $\mathbb{R}^{\Omega}$ such that $L=L_{\mu}$.
\end{prop}
\begin{proof}
Argue as in \cite[Theorem 4.11]{M2}.
\end{proof}
One knows that (\ref{s}) is also strictly weaker than (\ref{o}).
The exact relationship between (\ref{n}) and (\ref{s}) remains unclear.

\section{The support of the measure}

We turn now to the problem of describing the support of the measure. As one might expect, the results we obtain are sharpest 
when $\Omega$ is countable.

We begin with an extension of Haviland's theorem \cite{hav}, \cite[Theorem 3.1.2]{M}.
Note that for a closed set $Y \subseteq \mathbb{R}^{\Omega}$ the following are equivalent:

(i) $Y$ is described by countably many inequalities of the form $\hat{g}\ge 0$, $g \in A_{\Omega}$, i.e., $\exists$ a countable 
subset $S$ of $A_{\Omega}$ such that $Y = X_S= \{ \alpha \in \mathbb{R}^{\Omega} \mid \hat{g}(\alpha)\ge 0 \ \forall \ g\in S\}$.

(ii) $\exists$ a countable subset $I \subseteq \Omega$ and a closed subset $Y'$ of $\mathbb{R}^I$ such that $Y = \pi^{-1}(Y')$, 
where $\pi : \mathbb{R}^{\Omega} \rightarrow \mathbb{R}^I$ is the projection.

The equivalence of (i) and (ii) is a consequence of Proposition \ref{countable}. If $\Omega$ is countable then any closed subset 
$Y$ of $\mathbb{R}^{\Omega}$ satisfies these conditions.
\begin{prop}\label{hav}
Suppose $L : A_{\Omega} \rightarrow \mathbb{R}$ is linear and $L(\operatorname{Pos}_{A_{\Omega}}(Y)) \subseteq [0,\infty)$ where 
$Y$ is a closed subset of $\mathbb{R}^{\Omega}$ satisfying either of the equivalent conditions (i), (ii).
Then there exists a constructibly Radon measure $\nu$ on $\mathbb{R}^{\Omega}$ supported by $Y$ such that $L = L_{\nu}$.
\end{prop}
\begin{proof}By Proposition \ref{general extendibility} there exists an extension of $L$ to a linear functional $L$ on $B_{\Omega}$ 
such that $L(\operatorname{Pos}_{B_{\Omega}}(Y)) \subseteq [0,\infty)$. Denote by $\nu$ the constructibly Radon measure on 
$\mathbb{R}^{\Omega}$ corresponding to this extension. Fix a countable set $S$ in $A_{\Omega}$ such that $Y= X_S$. 
For each $g \in S$, choose $g' \in C_{\Omega}$ of the form $g' = g/p_g$ for some suitably chosen element 
$p_g = (1+x_{j_1}^2)^{e_1}\dots (1+x_{j_k}^2)^{e_k}$. Let $S' = \{ g' \mid g\in S\}$. 
Let $Q' =$ the quadratic module of $C_{\Omega}$ generated by $S'$, $Q =$ the quadratic module of $B_{\Omega}$ generated by $S$. 
Note that $Q$ is also the quadratic module in $B_{\Omega}$ generated by $S'$, and $Q' \subseteq Q \subseteq  \operatorname{Pos}_{B_{\Omega}}(Y)$, 
so $L'(Q') \subseteq [0,\infty)$ where $L' := L|_{C_{\Omega}}$. 
By \cite[Corollary 3.4]{M1} there exists a Radon measure $\mu$ on $X(C_{\Omega})$ supported by $X_{Q'}$ such that 
$L'(f) = \int \hat{f} d\mu$ $\forall$ $f \in C_{\Omega}$. Uniqueness implies that $\mu$ is the Radon measure on $X(C_{\Omega})$ defined 
in Proposition \ref{main}. Remark \ref{support remark} implies that $\nu$ is supported by $X_{Q'}\cap X(B_{\Omega}) =X_{Q} = X_S =Y$.
\end{proof}
Our next results extend \cite[Corollary 0.6]{M3} and \cite[Remark 0.7]{M3}.
\begin{prop}\label{support}
Suppose $L : A_{\Omega} \rightarrow \mathbb{R}$ is a PSD linear map, (\ref{n}) holds for each $j \in \Omega$,
and $L(M)\subseteq [0,\infty)$ for some quadratic module $M$ of $A_{\Omega}$ which is the extension of a quadratic module of $A_I$ 
for some countable $I \subseteq \Omega$.
Then the associated constructibly Radon measure $\nu$ on $\mathbb{R}^{\Omega}$ is supported by $X_M$.
\end{prop}
An earlier version of Proposition \ref{support} is proved already in \cite[Theorem 2.2]{L}.
\begin{proof}
Denote by $L : B_{\Omega} \rightarrow \mathbb{R}$ the PSD linear extension of $L$ defined by 
$L(f) := \int \hat{f} d\nu$ $\forall$ $f\in B_{\Omega}$. Arguing as in \cite[Theorem 0.5]{M3} one sees that 
$L(gh\overline{h})\ge 0$ $\forall$ $h \in B_{\Omega}\otimes \mathbb{C}$ (so, in particular, $L(gh^2)\ge 0$ $\forall$ $h\in B_{\Omega}$) 
$\forall$ $g\in M$.
Denote by $Q$ the extension of $M$ to $B_{\Omega}$. It follows that $L(Q)\subseteq [0,\infty)$. By hypothesis, $Q$ is the extension of 
a quadratic module $Q_0$ of $B_I$, so $Q' := Q\cap C_{\Omega}$ is the extension of $Q_0' := Q_0\cap C_I$, for some countable 
$I \subseteq \Omega$. Then $X_{Q'} = \pi^{-1}(X_{Q_0'})$, where $\pi : X(C_{\Omega}) \rightarrow X(C_I)$ denotes the restriction, 
so $X_{Q'}$ is constructibly Borel.  By \cite[Corollary 3.4]{M1} the Radon measure $\mu$ on $X(C_{\Omega})$ associated to 
$L' = L|_{C_{\Omega}}$ is supported by $X_{Q'}$, so, by Remark \ref{support remark}, $\nu$ is supported by $X_M = X_Q = X_{Q'}\cap X(B_{\Omega})$.
\end{proof}
For a quadratic module of the form $M = \sum A_{\Omega}^{\ 2}+J$, $J$ an  ideal of $A_{\Omega}$ one can weaken the hypothesis a bit.
\begin{prop}\label{zeros}
Suppose $L = L_{\nu}$ for some constructibly Radon measure $\nu$ on $\mathbb{R}^{\Omega}$ and $L(J) = \{ 0\}$ for some countably 
generated ideal $J$ of $A_{\Omega}$. Then $\nu$ is supported by $Z(J)$. 
Here, $Z(J) := \{ \alpha \in \mathbb{R}^{\Omega} \mid \hat{g}(\alpha)=0 \ \forall g\in J\}$.
\end{prop}
\begin{proof}
Let $M = \sum A_{\Omega}^{\ 2}+J$.  Since $L$ is PSD the hypothesis on $J$ is equivalent to $L(M)\subseteq [0,\infty)$. 
The extension of $M$ to $B_{\Omega}$ is $Q = \sum B_{\Omega}^2+JB_{\Omega}$, where $JB_{\Omega}$ denotes the extension of $J$ to 
$B_{\Omega}$. Extend $L$  to $B_{\Omega}$ in the obvious way, i.e., $L(f) = \int \hat{f} d \mu$ $\forall$ $f\in B_{\Omega}$. 
By the Cauchy-Schwartz inequality, for $g\in A_{\Omega}$, 
\[
	L(gh)=0 \ \forall \ h\in A_{\Omega} \ \Leftrightarrow \ L(g^2)=0 \ \Leftrightarrow \ L(gh)=0 \ \forall \ h\in B_{\Omega}.
\]
This implies $L(JB_{\Omega}) = \{0\}$, i.e., $L(Q) \subseteq [0,\infty)$. At this point everything is clear.
\end{proof}
A special feature of the following result is that the measure $\nu$ obtained is Radon (not just constructibly Radon).
\begin{prop}\label{radon}
Suppose $M$ is a quadratic module of $A_{\Omega}$ and there exists a countable subset $I$ of $\Omega$ such that the quadratic 
module $M\cap A_{\Omega \backslash I}$  of $A_{\Omega \backslash I}$ is archimedean. Suppose $L : A_{\Omega} \rightarrow \mathbb{R}$ 
is linear, $L(M) \subseteq [0,\infty)$, and (\ref{n}) holds for each $j \in I$. Then there exists a unique Radon measure $\nu$ on 
$\mathbb{R}^{\Omega}$ such that $L = L_{\nu}$. Moreover, $\nu$ is supported by $X_M$.
\end{prop}
Special cases:
\begin{itemize}
	\item[(i)] If $M$ is an archimedean quadratic module of $A_{\Omega}$ then Proposition \ref{radon} applies, taking $I = \emptyset$.
	\item[(ii)] If $\Omega$ is countable then Proposition \ref{radon} applies to any quadratic module $M$ of $A_{\Omega}$, taking $I=\Omega$ 
	(so $M\cap A_{\Omega\backslash I} = M\cap A_{\emptyset} = M\cap \mathbb{R} = \{ r \in \mathbb{R} \mid r\ge 0\}$, a quadratic module of 
	$\mathbb{R}$ which is obviously archimedean).
\end{itemize}
\begin{proof}
By hypothesis, there exists $N_j >0$ such that $N_j^2-x_j^2 \in M$ for each $j\in \Omega \backslash I$. It follows, e.g., using 
proposition \ref{jacobi}, that $L(x_j^{2k}) \le N_j^{2k}L(1)$, so (\ref{o}) holds for each $j\in \Omega \backslash I$. 
By Proposition \ref{exactly one} there exists a unique constructibly Radon measure $\nu$ on $\mathbb{R}^{\Omega}$ such that $L = L_{\nu}$. 
Extending $L$ to $B_{\Omega}$ in the obvious way and arguing as in Proposition \ref{support} we see that $L(Q) \subseteq [0,\infty)$ 
where $Q$ is the extension of $M$ to $B_{\Omega}$. By \cite[Corollary 3.4]{M1} there exists a Radon measure $\mu$ on $X(C_{\Omega})$ 
supported by $X_{Q\cap C_{\Omega}}$ such that $L(f) = \int \hat{f}_{C_{\Omega}}d\mu$ $\forall$ $f\in C_{\Omega}$. By Proposition \ref{main} 
$\mu$ is supported by the Borel set $E := X_{Q\cap C_{\Omega}} \backslash \cup_{j\in I} \Delta_j$. According to Proposition \ref{addendum} 
it suffices to show $E \subseteq X(B_{\Omega})$. But this is clear. Let $\alpha \in E$. If $j\in \Omega \backslash I$ then $N_j^2-x_j^2 \in Q$ 
so $\frac{1}{1+x_j^2}-\frac{1}{1+N_j^2} = \frac{N_j^2-x_j^2}{(1+x_j^2)(1+N_j^2)} \in Q\cap C_{\Omega}$. 
Thus $\alpha(\frac{1}{1+x_j^2}) \ge \frac{1}{1+N_j^2}$. If $j\in I$ then $\alpha \notin \Delta_j$, so $\alpha(\frac{1}{1+x_j^2})>0$.
\end{proof}
\section{Cylinder results}

Fix $i_0 \in \Omega$ and let $\Omega' := \Omega \backslash \{ i_0\}$. Consider the subalgebras 
$A_{\Omega} \subseteq B_{\Omega'}[x_{i_0}] \subseteq B_{\Omega}$ and $C_{\Omega'}[x_{i_0}] \subseteq C_{\Omega}$. 
Observe that $X(B_{\Omega'}[x_{i_0}])$ is naturally identified with $\mathbb{R}^{\Omega}$ and $X(C_{\Omega'}[x_{i_0}])$ is naturally
identified with $\mathbb{S}^{\Omega'}\times \mathbb{R}$.

The cylinder results in \cite[Section 4]{M2} and \cite{M3} extend in a straightforward way to the case where $\Omega$ is infinite. 
As a consequence, we are able to strengthen slightly the statement of Propositions \ref{exactly one}, \ref{Nussbaum} and \ref{strong}.
\begin{prop} \label{c1} \

(1) For $f \in C_{\Omega'}[x_{i_0}]$, $\hat{f} \ge 0$ on $\mathbb{S}^{\Omega'}\times \mathbb{R}$ iff $\exists$ $k\ge 0$ such that 
$f+\epsilon(1+x_{j_0}^2)^k \in \sum C_{\Omega'}[x_{j_0}]^2$ $\forall$ real $\epsilon>0$.

(2) For $f\in B_{\Omega'}[x_{i_0}]$, $\hat{f} \ge 0$ on $\mathbb{R}^{\Omega}$ iff $\exists$ $q$ of the form 
$q= \prod_{k=1}^n (1+x_{i_k}^2)^{\ell_k}$, where $x_{i_1},\dots,x_{i_n}$ are the variables appearing in the coefficients of $f$ and 
$k\ge 0$ such that $f+\epsilon q(1+x_{i_0}^2)^k \in \sum B_{\Omega'}[x_{i_0}]^2$ $\forall$ real $\epsilon >0$.
\end{prop}
\begin{proof}
(1) Since the quadratic module $\sum C_{\Omega'}^2$ of $C_{\Omega'}$ is archimedean, this is an immediate consequence of 
\cite[Theorem 5.1]{M1}. (2) If $f\in B_{\Omega'}[x_{i_0}]$, say $f\in B_I[x_{i_0}]$ where $I \subseteq \Omega'$ is finite, there 
exists an element $q$ of the form $q = \prod_{j\in I} (1+x_j^2)^{\ell_j}$ such that $\frac{f}{q}\in C_{\Omega'}[x_{i_0}]$. 
Thus, if $f \ge 0$ on $\mathbb{R}^{\Omega}$ then $\frac{f}{q}\ge 0$ on $\mathbb{S}^{\Omega'}\times \mathbb{R}$ so 
$\frac{f}{q}+\epsilon(1+x_{i_0}^2)^k \in \sum C_{\Omega'}[x_{i_0}]^2$ for some $k\ge 0$ and, consequently, 
$f+\epsilon q(1+x_{i_0}^2)^k \in \sum B_{\Omega'}[x_{i_0}]^2$ $\forall$ real $\epsilon >0$.
\end{proof}
\begin{prop}\label{c2}
For a linear functional $L : A_{\Omega} \rightarrow \mathbb{R}$ the following are equivalent:
\begin{itemize}
	\item[(1)] $L$ is a positive linear functional.
	\item[(2)] $L$ extends to a PSD linear functional $L : B_{\Omega} \rightarrow \mathbb{R}$.
	\item[(3)] $L$ extends to a PSD linear functional $L : B_{\Omega'}[x_{i_0}] \rightarrow \mathbb{R}$.
	\item[(4)] $\forall$ $f\in A_{\Omega}$ and $\forall$ $p$ of the form $p= \prod_{j=1}^n (1+x_{i_j}^2)^{\ell_j}$, where 
	$x_{i_1},\dots,x_{i_n}$ are the variables appearing in the coefficients of $f$ (viewing $f$ as a polynomial in $x_{i_0}$ with 
	coefficients in $A_{\Omega'}$)  and $\ell_j \ge 0$, $j=1,\dots,n$, $pf \in \sum A_{\Omega}^2$ $\Rightarrow$ $L(f)\ge 0$.
\end{itemize}
\end{prop}
\begin{proof} (1) $\Rightarrow$ (2). By Proposition \ref{extendibility}.
(2) $\Rightarrow$ (3). Immediate.
(3) $\Rightarrow$ (4). Since $pf \in \sum A_{\Omega}^2$, it follows that $p^2f \in \sum A_{\Omega}^2$, so $f \in \sum B_{\Omega'}[x_{i_0}]^2$. 
Since the extension of $L$ to $B_{\Omega'}[x_{i_0}]$ is PSD this implies $L(f)\ge 0$.
(4) $\Rightarrow$ (1). Suppose $f\in A_{\Omega}$, $\hat{f}\ge 0$ on $\mathbb{R}^{\Omega}$. By Proposition \ref{c1} (2) 
$\exists$ $q = \prod_{j=1}^n (1+x_{i_j}^2)^{\ell_j}$, where $x_{i_1},\dots,x_{i_n}$ are the variables appearing in the coefficients of $f$ 
and $k\ge 0$ such that $f+\epsilon q(1+x_{i_0}^2)^{k} \in \sum B_{\Omega'}[x_{i_0}]^2$ $\forall$ $\epsilon > 0$. Clearing denominators, 
$p^2(f+\epsilon q(1+x_{i_0}^2)^k)\in \sum A_{\Omega}^2$ for some $p$ (depending on $\epsilon$) of the form $p= \prod_{j=1}^n (1+x_{i_j}^2)^{m_j}$. 
By (4), $L(f)+\epsilon L(q(1+x_{i_0}^2)^k)=L(f+\epsilon q(1+x_{i_0}^2)^k)\ge 0$. Since $\epsilon>0$ is arbitrary, this implies $L(f)\ge 0$.
\end{proof}
\begin{prop}\label{c3}
Suppose $L : A_{\Omega} \rightarrow \mathbb{R}$ is linear and PSD and condition (\ref{n}) holds, for each $j \in \Omega$, $j \ne i_0$.
Then there exists a constructibly Borel measure $\nu$ on $\mathbb{R}^{\Omega}$ such that $L = L_{\nu}$. If condition (\ref{n}) also holds 
for $j=i_0$ then $\nu$ is unique.
\end{prop}
\begin{proof}
Argue as in \cite[Corollary 4.7 and 4.8]{M2} and \cite[Theorem 0.1]{M3}.
\end{proof}
Combining Proposition \ref{c3} and Remark \ref{clarify} yields the following result which is due to Nussbaum \cite{N} in case $|\Omega|<\infty$.
\begin{prop}\label{c4}
Suppose $L : A_{\Omega} \rightarrow \mathbb{R}$ is linear and PSD and, for each $j \in \Omega$, $j\ne i_0$ the Carleman condition 
(\ref{o}) holds.
Then there exists a constructibly Radon measure $\nu$ on $\mathbb{R}^{\Omega}$ such that $L = L_{\nu}$. If condition (\ref{o}) also 
holds for $j=i_0$ then $\nu$ is unique.
\end{prop}
Proposition \ref{strong} extends in a similar way.
\begin{prop}\label{c5}
Suppose $L : A_{\Omega} \rightarrow \mathbb{R}$ is linear and PSD. For each $j\in \Omega$ fix a Radon measure $\mu_j$ on $\mathbb{R}$ 
such that $L|_{\mathbb{R}[x_j]} = L_{\mu_j}$ and suppose, for each $j \in \Omega$, $j \ne i_0$ condition (\ref{s}) holds.
Then there exists a constructibly Radon measure $\mu$ on $\mathbb{R}^{\Omega}$ such that $L=L_{\mu}$. If condition (\ref{s}) also holds 
for $j=i_0$ then $\nu$ is unique.
\end{prop}

\end{document}